\documentclass[a4paper,11pt]{amsart}
\usepackage{graphicx} 
\usepackage{amsfonts, amssymb, amsmath}
\usepackage{amsthm}
\usepackage{mathtools}
\usepackage[margin=1in]{geometry}
\usepackage{enumerate}
\usepackage[all]{xy}
\usepackage{color, xcolor}

\usepackage{tikz, tikz-cd}
\usetikzlibrary{3d}
\usepackage{caption}
\usepackage{subcaption}
\usepackage[all]{xy}

\theoremstyle{definition}

\newtheorem{Thm}{Theorem}[section]
\newtheorem{Prop}[Thm]{Proposition}

\usepackage{epic}

\begin{document}

\title{Degree of irrationality of a product of two elliptic curves}
\author{Yongnam Lee}

\address{Center for Complex Geometry, Institute for Basic Science (IBS), 55 Expo-ro, Yuseong-gu, Daejeon, 34126 Korea}
\email{ynlee@ibs.re.kr}


\date{\today}

\begin{abstract}
In this short paper, we prove that, for any two complex elliptic curves \(E\) and \(F\), the
product \(E\times F\) has degree of irrationality \(3\). The upper bound
is obtained from compatible cyclic plane-cubic models of \(E\) and
\(F\): three diagonal bihomogeneous sections define a dominant rational
map \(E\times F\dashrightarrow\mathbb P^2\) of degree \(3\). The lower
bound follows by excluding dominant rational maps of degrees \(1\) and
\(2\) from an abelian surface to \(\mathbb P^2\).
\end{abstract}

\maketitle

\section{Introduction}

Let \(X\) be an irreducible projective variety of dimension \(n\). Its
\emph{degree of irrationality} is
\[
  \operatorname{irr}(X)
  =\min\bigl\{\deg(f)\,\bigm|\,
  f\colon X\dashrightarrow\mathbb P^n
  \text{ is dominant}\bigr\}.
\]
When \(X\) is a curve, this invariant is the gonality of its
normalization. More generally, measures of irrationality provide
numerical invariants that quantify how far a variety is from being
rational. Although they are natural birational invariants, they are
difficult to compute in general.

One basic open problem asks whether the degree of irrationality can
increase under specialization in a smooth one-parameter family. Chen
and Esser proved that it cannot increase in several cases
\cite{CE}. In particular, they asked whether there exist
non-isogenous elliptic curves \(E\) and \(F\) such that
\(\operatorname{irr}(E\times F)=4\); see
\cite[Example~3.2]{CE}. Chen and Martin similarly highlighted the
open problem of deciding whether the product of two very general
elliptic curves has degree of irrationality \(3\) or \(4\)
\cite{CM26}. Contrary to the expectation that degree \(4\) should
occur, we prove the following.

\begin{Thm}\label{degree3}
For every pair of complex elliptic curves \(E\) and \(F\),
\[
  \operatorname{irr}(E\times F)=3.
\]
\end{Thm}

The same conclusion was obtained independently by Huang, Jiang, and
Yang as part of their classification of complex abelian surfaces of
degree of irrationality \(3\) \cite{HJY}. They prove that a
complex abelian surface admits a dominant rational map of degree \(3\)
to \(\mathbb P^2\) if and only if it admits a polarization of type
\((1,1)\), \((1,2)\), \((1,3)\), or \((1,6)\). Earlier, Yoshihara
constructed degree-\(3\) maps for certain self-products of elliptic
curves with complex multiplication \cite{Yosh}.

We briefly describe the proof. The lower bound
\(\operatorname{irr}(E\times F)\geq 3\) follows from the fact that an
abelian surface is not rational and cannot admit a dominant rational
map of degree \(2\) to \(\mathbb P^2\). For the upper bound, we use the
cyclic plane-cubic family
\[
  C_\lambda\colon
  x_0^2x_1+x_1^2x_2+x_2^2x_0+\lambda x_0x_1x_2=0.
\]
Its \(j\)-invariant is \(j(C_\lambda)
=-\frac{\lambda^3(\lambda^3+24)^3}{\lambda^3+27},\)
and every complex \(j\)-invariant occurs in this family. After choosing
compatible cyclic models for \(E\) and \(F\), the three sections
\(x_0z_0,x_1z_1,x_2z_2\) define a rational map
\(E\times F\dashrightarrow\mathbb P^2.\)
Its base locus consists of three double base points and three simple
base points. Blowing them up gives a base-point-free divisor \(M\) with
\[
  M^2=18-3\cdot2^2-3\cdot1^2=3,
\]
so the resulting map is dominant of degree \(3\).

Section~2 presents a motivating explicit example and describes the
associated exceptional Pompilj triple plane and its branch 
divisor. In Section~3, we prove Theorem~\ref{degree3} and show that, for
non-isogenous factors, the degree \(3\) function field extension arising
from our construction is not Galois. Throughout, we work over the field
of complex numbers.

\subsection*{Acknowledgements}
The author was supported by the Institute for Basic Science
(IBS-R032-D1) and thanks the Department of Mathematics at the National
University of Singapore for its hospitality during his visit. He is
grateful to Carl Lian and De-Qi Zhang for helpful comments, and to Carl
Lian for sharing another example of a product \(E\times F\) with
\(\operatorname{irr}(E\times F)=3\). The initial explicit example in
Section~2 emerged from discussions with ChatGPT.

\section{Example}
 In this section, we give an example of two non-isogenous elliptic curves $E$ and $F$ such that the degree of irrationality of $E\times F$ is three. 
 We illustrate this example because it provides a key idea for the proof of the theorem. 

We take
\[
E:\quad x_0^2x_1+x_1^2x_2+x_2^2x_0=0,
\]so \(j(E)=0\). Let \(\alpha\in\mathbb C\) satisfy
\(
\alpha^3=-18+6\sqrt3,
\) and take
\[
F:\quad
z_0^2z_2+z_2^2z_1+z_1^2z_0+
\alpha z_0z_1z_2=0.
\] Then \(j(F)=1728\),  and both cubics are smooth. Moreover,
\[
\operatorname{End}^0(E)\cong\mathbb Q(\sqrt{-3}),
\qquad
\operatorname{End}^0(F)\cong\mathbb Q(i),
\] so \(E\) and \(F\) are not isogenous.

\begin{Prop}
The rational map \[
E\times F\dashrightarrow\mathbb P^2,\qquad
(x,z)\longmapsto[x_0z_0:x_1z_1:x_2z_2]
\]has degree
\(
18-\bigl(3\cdot2^2+3\cdot1^2\bigr)=3.
\)
\end{Prop}

\begin{proof}
Let
\[
S=E\times F,\qquad
L=\mathcal O_E(1)\boxtimes\mathcal O_F(1),
\]and consider the three sections
\( s_i=x_i z_i\in H^0(S,L),\ i=0,1,2.\)

Write
\(
H_E=\mathcal O_E(1),
H_F=\mathcal O_F(1)
\). Let $p_E$ and $p_F$ be projections from $E\times F$ to $E$ and $F$, respectively. The line bundles \(H_E\) and \(H_F\) both have degree 3. Then, on \(E\times F\),
\[
L^2
=2(p_E^*H_E\cdot p_F^*H_F)
=2\deg(H_E)\deg(H_F)
=2\cdot3\cdot3
=18.
\]Thus two general linear combinations of \(s_0,s_1,s_2\) intersect in \(18\) points, counted with multiplicities, including the fixed base point contributions.

This linear system has six base points. Let
\[
e_0=[1:0:0],\quad e_1=[0:1:0],\quad e_2=[0:0:1].
\]For
\[
E:\quad x_0^2x_1+x_1^2x_2+x_2^2x_0+\lambda x_0x_1x_2=0,
\]the coordinate line intersections are
\((x_i=0)|_E=e_{i+1}+2e_{i+2}.\) Equivalently, at \(e_i\),
\(\operatorname{ord}_{e_i}(x_{i+1})=2, 
\operatorname{ord}_{e_i}(x_{i+2})=1.\) For the oppositely oriented cubic
\[
F:\quad z_0^2z_2+z_2^2z_1+z_1^2z_0+\mu z_0z_1z_2=0,
\]we have
\( \operatorname{ord}_{e_i}(z_{i-1})=2,
\operatorname{ord}_{e_i}(z_{i+1})=1.\) 

The simultaneous equations
\[
x_0z_0=x_1z_1=x_2z_2=0
\]give exactly the six points
\( (e_i,e_j) \text{ with}\  i\ne j.\) They divide into two kinds:
\[
p_i=(e_i,e_{i+1}),\qquad
q_i=(e_i,e_{i-1}).
\]

Among the six base points, there are three double base points.
Consider
\(p_i=(e_i,e_{i+1}).\) Choose local parameters \(u\) on \(E\) at \(e_i\) and \(v\) on \(F\) at \(e_{i+1}\). Up to units,
\(x_{i+1}\sim u^2, x_{i+2}\sim u,\) and
\(z_i\sim v^2, z_{i+2}\sim v.\) Meanwhile \(x_i\) and \(z_{i+1}\) are units. Hence
\[
s_i=x_i z_i\sim v^2, \
s_{i+1}=x_{i+1}z_{i+1}\sim u^2, \
s_{i+2}=x_{i+2}z_{i+2}\sim uv.
\]
where all subscripts in this computation are taken modulo 3. Thus 
\[
I_{p_i}=(u^2,uv,v^2)=\mathfrak m_{p_i}^2.\] 
This is a base point of multiplicity \(2\). Its contribution to the intersection of two general members is \(2^2=4.\) After blowing up \(p_i\), the leading terms on the exceptional line are \([u^2:uv:v^2].\)
They have no common zero on \(\mathbb P^1\), so there is no infinitely near base point.

The other three base points are simple. At
\(q_i=(e_i,e_{i-1}),
\) two of the sections have independent linear terms. In suitable local parameters,
\(I_{q_i}=(u,v)=\mathfrak m_{q_i}.\) Thus \(q_i\) is a simple base point and contributes \(1\). Again, no further base point lies above \(q_i\).

Let \(\pi:\widetilde S\longrightarrow S\) be the blowup of these six points, with exceptional curves \(P_i\) over \(p_i\) and \(Q_i\) over \(q_i\). 
For the resulting morphism \(\widetilde\phi:\widetilde S\longrightarrow\mathbb P^2,\) we let
\(M=\widetilde\phi^*\mathcal O_{\mathbb P^2}(1).\) Then
\[
M=\pi^*L-2\sum_{i=0}^2P_i-\sum_{i=0}^2Q_i.
\]Using \(P_i^2=Q_i^2=-1\) and the orthogonality of distinct exceptional divisors,
\(M^2= L^2-3\cdot2^2-3\cdot1^2=3.\) Since \(M^2=3>0\), its image cannot be a curve:
Since \(M=\widetilde\phi^*\mathcal O_{\mathbb P^2}(1)\), a
one-dimensional image would force \(M^2=0\). Thus \(M^2>0\) implies that
the image is two-dimensional.

Hence it maps dominantly onto \(\mathbb P^2\). Therefore
\(M^2=\deg(\widetilde\phi)\cdot
\bigl(\mathcal O_{\mathbb P^2}(1)^2\bigr)
=\deg(\widetilde\phi).\) Hence \(\deg\phi=3.\)
\end{proof}

\begin{Prop}
For the exceptional Pompilj triple-plane model associated with this construction \cite{CM25}, the degree of the branch divisor $B$ of this example is 18. 
And \(B=2\Delta+B_{12}\) where \(\Delta\) is the coordinate triangle and \(B_{12}\) is a reduced degree \(12\) curve.
\end{Prop}

\begin{proof}
Let
\[
S=E\times F,\qquad
\phi:S\dashrightarrow\mathbb P^2_y,\qquad
(x,z)\longmapsto[y_0:y_1:y_2]=[x_0z_0:x_1z_1:x_2z_2].
\]Blow up the three double base points \(p_i\) and three simple base points \(q_i\): $\pi:\widetilde S\longrightarrow S$.

The resolved linear system is defined by the divisor
\[M=\pi^*L-2\sum_{i=0}^2P_i-\sum_{i=0}^2Q_i,
\qquad M^2=3,\]
giving a generically finite morphism $\widetilde\phi:\widetilde S\longrightarrow\mathbb P^2$ of degree \(3\).
This morphism is not finite. For instance, the strict transforms of
$E\times\{e_i\}, \{e_i\}\times F$ are contracted to coordinate points of \(\mathbb P^2\). Therefore the triple plane model is its Stein factorization:
\[
\widetilde S\xrightarrow{\rho}X
\xrightarrow{\pi}\mathbb P^2,
\]where \(X\) is normal, \(\rho\) is birational and \(\deg\pi=3.\)

Fix \(x=[x_0:x_1:x_2]\in E\). The fiber \(\{x\}\times F\) maps to a plane cubic \(C_x\subset\mathbb P^2_y\).
Substitute \(z_i=y_i/x_i\) into
\[
z_0^2z_2+z_2^2z_1+z_1^2z_0+\alpha z_0z_1z_2=0.
\]Set
\[
A=x_0^2x_1,\qquad
B=x_1^2x_2,\qquad
C=x_2^2x_0,\qquad
D=x_0x_1x_2.
\]Then
\[
C_x:\quad
B\,y_0^2y_2+A\,y_2^2y_1+C\,y_1^2y_0
+\alpha D\,y_0y_1y_2=0.
\]
Because \(x\in E\), \(A+B+C=0\) and \(ABC=D^3.\) Eliminating \(C=-A-B\), we obtain
\[
C_x=
A\bigl(y_2^2y_1-y_1^2y_0\bigr)
+B\bigl(y_0^2y_2-y_1^2y_0\bigr)
+\alpha D\,y_0y_1y_2.
\]Thus the curves \(C_x\) lie in the net
\[
\mathcal N=
\left\langle
y_2^2y_1-y_1^2y_0,\ 
y_0^2y_2-y_1^2y_0,\ 
\alpha y_0y_1y_2
\right\rangle
\subset |\mathcal O_{\mathbb P^2}(3)|.
\]The parameter point \([A:B:D]\) lies on the plane cubic
\[
\Gamma:\quad D^3+AB(A+B)=0.
\]This is a smooth cubic isomorphic to \(E\):
On a dense open set, the map
\[
[x_0:x_1:x_2]\longmapsto[A:B:D]
=[x_0^2x_1:x_1^2x_2:x_0x_1x_2]
\]
has inverse \([A:B:D]\longmapsto[AD:AB:D^2].\)
Thus it is birational. Since both curves are smooth projective cubics, it extends to an isomorphism. 

Evaluation of the above net defines a rational map
\(\psi:\mathbb P^2_y\dashrightarrow(\mathbb P^2_{\mathcal N})^\vee\) given by cubic forms:
\[
\psi(y)=
\left[
y_2^2y_1-y_1^2y_0:
y_0^2y_2-y_1^2y_0:
\alpha y_0y_1y_2
\right].
\]For a general \(y\), the cubics in \(\mathcal N\) containing \(y\) form a line
\(\ell_y\subset\mathbb P^2_{\mathcal N}.\) The three points of
\(\ell_y\cap\Gamma\) correspond exactly to the three points in the fiber of \(\phi\) over \(y\).
The triple plane of \(\Gamma\) is
\[
\Gamma^{(2)}
\longrightarrow
(\mathbb P^2_{\mathcal N})^\vee,
\]sending an effective divisor \(x_1+x_2\) to the line through \(x_1,x_2\). Thus \(X\to\mathbb P^2_y\) is the normalization of the pullback of the basic exceptional Pompilj triple plane.

The branch curve of
\[
\Gamma^{(2)}\longrightarrow(\mathbb P^2_{\mathcal N})^\vee
\]is the dual curve \(\Gamma^\vee\). Since \(\Gamma\) is a smooth plane cubic,
\(\deg\Gamma^\vee=3(3-1)=6.\) The pullback map \(\psi\) is defined by cubic forms, so
\(\psi^*\mathcal O(1)=\mathcal O_{\mathbb P^2}(3).\) Therefore the branch divisor of \(X\to\mathbb P^2\) is
\(B=\psi^{-1}(\Gamma^\vee),\) and \(\deg B=3\cdot6=18.\)

We have \(K_{\widetilde S}\cdot M=3\cdot2+3\cdot1=9.\) A general \(C\in|M|\) has
\[
2g(C)-2=M^2+M\cdot K_{\widetilde S}=3+9=12,
\]so \(g(C)=7\). The degree three map \(C\to\mathbb P^1\) has total ramification
\[
2g(C)-2-3(-2)=12+6=18,
\]again giving \(\deg B=18\), counted as the discriminant branch divisor.

Recall \(\Gamma: D^3+AB(A+B)=0
\subset\mathbb P^2_{[A:B:D]}.\) Let \([U:V:W]\) be dual coordinates. An equation of the dual sextic \(\Gamma^\vee\) is
\[
\begin{aligned}
\Phi(U,V,W)
={}&\bigl(W^3-3(U+V)UV\bigr)^2\\
&-4(U+V)\bigl(3UV-(U+V)^2\bigr)
   \bigl(W^3-3(U+V)UV\bigr)\\
&+12UV\bigl(3UV-(U+V)^2\bigr)^2.
\end{aligned}
\]
And \(\psi:\mathbb P^2_y\dashrightarrow(\mathbb P^2_{\mathcal N})^\vee\) with
\(U=y_1(y_2^2-y_0y_1), \ V=y_0(y_0y_2-y_1^2),\ W=\alpha y_0y_1y_2.\) Consequently, the branch divisor
\[
B: \quad
\Phi\bigl(U(y),V(y),W(y)\bigr)=0.
\]Direct substitution gives the factorization
\[
\Phi\bigl(U(y),V(y),W(y)\bigr)
=(y_0y_1y_2)^2Q_{12}(y_0,y_1,y_2),
\]where \(Q_{12}\) is a homogeneous polynomial of degree \(12\). We set \(B_{12}: Q_{12}=0\).
It shows that 
\(B=2\Delta+B_{12}\) where \(\Delta\) is the coordinate triangle and \(B_{12}\) is a reduced degree \(12\) curve.
A direct square-freeness computation shows that \(Q_{12}\) is reduced.
For instance, after reduction modulo the prime
\((13,\sqrt3-4)\subset\mathbb Z[\sqrt3]\), the polynomial retains degree
12 and is relatively prime to its three first partial derivatives.
\[
\gcd\!\left(
\overline Q_{12},
\frac{\partial\overline Q_{12}}{\partial y_0},
\frac{\partial\overline Q_{12}}{\partial y_1},
\frac{\partial\overline Q_{12}}{\partial y_2}
\right)=1
\quad\text{in }\mathbb F_{13}[y_0,y_1,y_2].
\]
Hence \(Q_{12}\) is squarefree in characteristic zero. 
\end{proof}

\section{Proof of the theorem}

The following proposition is well known to experts; we include its proof for the reader’s convenience.
\begin{Prop}
Let $E$ and $F$ be elliptic curves. Then \(\operatorname{irr}(E\times F)\ge 3.\)
\end{Prop}

\begin{proof}
Let \(A=E\times F.\) Since \(A\) is an abelian surface, it is not rational, so \(\operatorname{irr}(A)\ne1\). It remains to rule out degree \(2\).
Assume there is a dominant rational map
\[
f:A\dashrightarrow\mathbb P^2
\]of degree \(2\). In characteristic \(0\), the corresponding quadratic extension
\( \mathbb C(\mathbb P^2)\subset \mathbb C(A)\) is Galois. Let \(\iota\) be the nontrivial Galois involution. Then
\(
A/\langle\iota\rangle\sim_{\mathrm{bir}}\mathbb P^2.
\) Every birational self-map of an abelian surface is regular, so \(\iota\in\operatorname{Aut}(A)\). Consider its action on
\[
V:=H^0(A,\Omega_A^1),
\qquad \dim V=2.
\]Since the quotient is rational, it has no holomorphic one-forms. Hence
\(
V^\iota=0.
\) Because \(\iota^2=1\), its eigenvalues on \(V\) are \(\pm1\). Having no invariant vector forces
\(
\iota^*|_V=-\operatorname{id}_V.
\) 

Write \(\iota=t_a\circ\alpha\), where \(t_a\) is a translation and
\(\alpha\) is a group automorphism of \(A\). Since translations act
trivially on holomorphic one-forms and a group endomorphism of a complex
abelian variety is determined by its differential, the equality
\(\iota^*|_V=-\mathrm{id}_V\) implies \(\alpha=[-1]\). Thus
\(\iota=t_a\circ[-1]\), whose fixed locus is finite (it is given by
\(2x=a\)). Consequently, \(A/\langle\iota\rangle\) has only quotient
singularities of type \(\frac12(1,1)\). The invariant holomorphic
two-form descends to the smooth locus of the quotient and extends to its
minimal resolution. This contradicts
\(
A/\langle\iota\rangle\sim_{\mathrm{bir}}\mathbb P^2,
\) since a rational surface has \(p_g=0\).
Therefore no degree-\(2\) rational map \(A\dashrightarrow\mathbb P^2\) exists. 
\end{proof}

\begin{proof}[Proof of Theorem~\ref{degree3}]
Consider the family of cubics
\[
C_\lambda:\quad
x_0^2x_1+x_1^2x_2+x_2^2x_0+\lambda x_0x_1x_2=0.
\]
It is smooth exactly when \(\lambda^3\neq-27\), and
\[
j(C_\lambda)=
-\frac{\lambda^3(\lambda^3+24)^3}{\lambda^3+27}.
\]
This family contains a representative of every isomorphism class of elliptic curves. Indeed, for a prescribed \(j_0\in\mathbb C\), choose a root \(t\) of
\[
t(t+24)^3+j_0(t+27)=0.
\]This quartic always has a complex root, and \(t\neq-27\), since its value at \(-27\) is \(729\). Taking \(\lambda^3=t\) gives a smooth \(C_\lambda\) with \(j(C_\lambda)=j_0\).
Consequently, after choosing suitable plane models, arbitrary \(E\) and \(F\) may be presented as 
\[
E=C_\lambda,\qquad
F=C_\mu^{\mathrm{op}}:
z_0^2z_2+z_2^2z_1+z_1^2z_0+\mu z_0z_1z_2=0,
\]because exchanging \(z_1\) and \(z_2\) identifies the opposite family with the same cyclic family.
Now consider the rational map \(f\):
\[
f:E\times F\dashrightarrow\mathbb P^2,
\qquad
(x,z)\longmapsto[x_0z_0:x_1z_1:x_2z_2].
\]The base locus consists of six points
\(p_i=(e_i,e_{i+1}),\ q_i=(e_i,e_{i-1}).\) At \(p_i\), the local base ideal is
\[
(u^2,uv,v^2)=\mathfrak m_{p_i}^2,
\]whereas at \(q_i\) it is \((u,v)=\mathfrak m_{q_i}.\) These descriptions are independent of \(\lambda,\mu\).  Blowing up once at each of these points resolves the base locus, with no infinitely near base points.
Writing \(L=\mathcal O_E(1)\boxtimes\mathcal O_F(1)\), we have \(L^2=18\). The divisor defining the resolved linear system is
\[
M=\pi^*L-2\sum_{i=0}^2P_i-\sum_{i=0}^2Q_i,
\]and hence
\[
M^2=18-3\cdot2^2-3\cdot1^2=3.
\]Because the resolved system is base point free and \(M^2>0\), its image is two-dimensional, so
\[
\deg f=M^2=3.
\]Thus \(\operatorname{irr}(E\times F)\le3\).
On the other hand, the previous proposition shows \(\operatorname{irr}(E\times F)\ge 3.\) This proves the theorem.
\end{proof}

\begin{Prop}
Let \(E_1\) and \(E_2\) be non-isogenous complex elliptic curves and let
\(A=E_1\times E_2.\) Then the extension \(\mathbb C(\mathbb P^2)\subset\mathbb C(A)\) induced by the map \(f\) constructed in the proof of Theorem 3.2 is not Galois.
\end{Prop}

\begin{proof}
By the proof of Theorem~\ref{degree3}, the map \(f\) has degree \(3\). Suppose it were Galois. Its Galois group is \(G\simeq\mathbb Z/3\mathbb Z\), acting birationally on \(A\). Every birational self-map of an abelian variety is regular, so \(G\subset\operatorname{Aut}(A)\).
Because \(E_1\) and \(E_2\) are non-isogenous,
\[
\operatorname{Hom}(E_1,E_2)=\operatorname{Hom}(E_2,E_1)=0.
\]Consequently, every automorphism of \(A\) has the form
\[
g(x,z)=\bigl(\alpha_1(x)+a_1,\alpha_2(z)+a_2\bigr),
\]where \(\alpha_i\) is a group automorphism of \(E_i\).
Since \(A/G\) is birational to \(\mathbb P^2\),
\(
H^0(A,\Omega_A^1)^G=0.
\) Translations act trivially on holomorphic forms. For \(i=1,2\), the automorphism \(\alpha_i^*\) does not act as the identity on
\(H^0(E_i,\Omega_{E_i}^1).
\) Both must act by a primitive cube root of unity. Thus both \(E_1\) and \(E_2\) admit an origin preserving automorphism of order \(3\), forcing
\(
j(E_1)=j(E_2)=0.
\) Over \(\mathbb C\), they are therefore isomorphic, contradicting non-isogeny. 
\end{proof}

\end{document}